\documentclass{amsart}

\usepackage{amsmath}

\theoremstyle{plain}

\newtheorem{proposition}{Proposition}[section]

\numberwithin{equation}{section}

\theoremstyle{definition}

\title{A Geometric Representation}
\author{\emph{Nicholas Phat Nguyen} (nicholas.pn@gmail.com)}

\begin{document}

\maketitle

\begin{abstract}
This article provides a geometric representation for the well-known
isomorphism between the special orthogonal group of an isotropic quadratic space
of dimension 3 and the group of projective transformations of a projective line.
This geometric representation depends on the theory of inversive transformations
in dimension 1 as outlined in the 2021 article \emph{Projective Line Revisited} by the same author.  This geometric representation also provides a new perspective on some classical properties of the projective line, such as the classical cross ratio.


\emph{AMS2020 Subject Classification}:  51N30, 51N15, 51F25

\end{abstract}

\section{Introduction}
\label{sec:introduction}

It has been known since the 1950s that the special orthogonal group of an isotropic quadratic space of dimension 3 is isomorphic to the group of projective transformations of a projective line. That is discussed, for example, near the end of the classic 1957 book \emph{Geometric Algebra} by Emil Artin. See \cite{artin} at Theorem 5.20. (The Artin book has been republished in a 2016 edition by Dover.)  See also \cite{shimura} at section 25. Specifically, let $V$ be a vector space of dimension 3 over a commutative coordinate field $K$ of characteristic $\neq 2$, and suppose that $V$ is endowed with a regular quadratic form. If the quadratic space
$V$ is isotropic, i.e., if $V$ has an isotropic vector, then the special orthogonal group $SO(V)$ is isomorphic to the group $PGL(1, K)$ of projective transformations of the projective line $K\cup \{\infty\}$.

However, that isomorphism is based on general algebraic considerations of the associated Clifford algebra, and by itself has no geometric representation. Specifically, the theory of Clifford algebra in dimension 3 tells us that the special orthogonal group of $V$ is isomorphic to the group of invertible elements of the even Clifford algebra modulo scalar factors. The even Clifford algebra in this case is a quaternion algebra (central simple of dimension 4), and in particular must be isomorphic to a matrix algebra because the quadratic space is isotropic. So the special orthogonal group is isomorphic to the group of invertible 2-by-2 matrices modulo scalar factors, which is $PGL(1, K)$.

To my knowledge, there has been no published account of any geometric representation
for the above equivalence. By geometric representation, I mean a projective line constructed from the given quadratic space $V$ and a faithful action of the group $SO(V)$ on that line which is the same as the group of projective transformations. The goal of this article is to provide such a geometric representation, by the theory of inversive transformations in dimension 1 as outlined in my earlier article \emph{Projective Line Revisited}. See \cite{phatnguyen}.

For the convenience of the reader, we summarize here the basic ideas of our geometric representation.  Given an isotropic quadratic space $V$ of dimension 3, we consider the projective quadric defined by the quadratic form of $V$.  The group $SO(V)$ is naturally isomorphic to the projective orthogonal group of $V$, and its action on the quadric is naturally the same as the faithful action by the projective orthogonal group on the quadric, which is generated by reflections defined by nonisotropic vectors of $V$. 

There is a natural way for us to identify that quadric with the points of the projective line $K\cup \{\infty\}$ = $\mathbf{P}(K^2)$, and to identify nonisotropic vectors of $V$ with regular quadratic forms on $K^2$.  Under that identification, we will show that the action on the projective quadric defined by a nonisotropic vector of $V$ is the same as the involution or inversive transformation defined by the associated quadratic form on $\mathbf{P}(K^2)$.  The action of such inversive transformations has been defined and analyzed in \cite{phatnguyen}.

Accordingly, by using the theory of inversive transformations in dimension 1 as outlined in my earlier paper \emph{Projective Line Revisited}, we can identify the action of the special orthogonal group on the projective quadric of $V$ with the projective transformations of a projective line $K\cup \{\infty\}$.

With this geometric representation, we also gain a new perspective on some classical properties of the projective line, which we will discuss at the end of the article.

\section{Isotropic Quadratic Space of Dimension 3}
\label{sec:involutions}

Let $V$ be a $K$-vector space of dimension 3 over a commutative coordinate field $K$ of characteristic $\neq 2$, with a given regular quadratic form $\phi$. 

If $V$ is isotropic, then the quadratic space $V$ can be expressed as a direct sum of $<d>$ and an Artinian plane (also known as a hyperbolic plane), where $<d>$ denotes the quadratic form on $K$ given by $x \mapsto dx^2$, i.e., the associated bilinear pairing is $d$ times ordinary multiplication.  We can always choose a suitable orthogonal basis so that the quadratic form  $\phi$ on $V$ can be put in the diagonal form $<d, d, -d>$ = $d<1, 1, -1>$.  For any nonzero scalar $e$, the form $\phi$ and the form $e\phi$ have the same projective quadric, and the groups $SO(\phi)$ and $SO(e\phi)$ are naturally isomorphic, with the same action on the quadric.  See, e.g., \cite{shimura} at Lemma 24.9.  So for our purpose, we can assume that the quadratic form on $V$ is isomorphic to the form $<1, 1, -1>$, which is the sum of $K$ with ordinary multiplication and an Artinian plane.

(If interested, the reader can work directly with the form $<d, d, -d>$ by making necessary adjustments in relevant formulas.  However, in order to keep our references and formulas simple, we will work with the standard isotropic form $<1, 1, -1>$.)

The space $V$ with the standard isotropic form $<1, 1, -1>$ is isomorphic to the space of all symmetric bilinear forms on $K^2$, with a suitable pairing as described below.

To be specific, consider the set $E$ of all polynomials $p(X)$ of degree $\le 2$,
with coefficients in the field $K$. The set $E$ is naturally a $K$-vector space
of dimension~3. We will refer to a nonzero polynomial $p$ as a $2$-cycle,
$1$-cycle, or $0$-cycle depending on whether the degree of $p$ is $2$, $1$, or
$0$. For convenience, we will write each element $p$ of $E$ in the same form
$p(X) = aX^2 + bX + c$, with the understanding that each coefficient $a$, $b$,
and $c$ could be zero.

We can endow the vector space $E$ with a symmetric bilinear form as follows: 
if $p = aX^2+ bX + c$ and $p^\ast = a^\ast X^2+ b^\ast X + c^\ast$, then 
we define 
$$
\langle p, p^\ast\rangle =  bb^\ast - 2ac^\ast - 2a^\ast c.  
$$

We refer to this pairing as the cycle pairing or cycle product on $E$.  The cycle pairing is plainly isomorphic to the sum of $K$ (represented by the middle coefficient, with ordinary multiplication) and an Artinian plane. 

We can think of $E$ as the space of all symmetric bilinear forms on $K^2$. Specifically, a cycle $p$ of $E$ can be thought of as a function from $K$ to $K$ given by $p(x) = q((x, 1), (x, 1))$ where $q$ is a symmetric bilinear form on $K^2$.\footnote{ The elements $p$ of $E$ are defined as polynomials of degree~2 or less, but because the field $K$ has 3 or more elements, such a polynomial $p$ can be identified with a polynomial function from $K$ to $K$.} 

If $p = aX^2 + bX + c$, then the matrix for the corresponding symmetric bilinear form $q$ (relative to the standard basis of $K^2$) has entries $a$ and $c$ in the main diagonal, and entries $b/2$ in the cross diagonal. In other words, we can think of $p = aX^2 + bX + c$ as an abbreviated expression for the homogeneous polynomial $aX^2 + bXY + cY^2$.

The space $E$ with the cycle pairing is isometric to $V$, and so we can restrict our attention in the following to $E$.

\begin{proposition}
  \label{prop:1}
  The isotropic cycles of $E$ (with respect to the cycle pairing defined above) are the $0$-cycles and the $2$-cycles with zero discriminant $b^2 - 4ac$.  To each such isotropic cycle, we can associate a nonzero vector $(x, y)$ in $K^2$, unique up to scalar factors, that represents the projective zero point of that isotropic cycle (regarded as a quadratic form on $K^2$).
\end{proposition}  

\begin{proof}
  The $0$-cycles $p = c$ are clearly isotropic.  The $1$-cycles $p = bX + c$ with $b \neq 0$ are never isotropic as the norm $\langle p, p \rangle$ of such a cycle is $b^2$. For a $2$-cycle $p = aX^2 + bX + c$ with $a \neq 0$, its norm $\langle p, p \rangle = b^2 - 4ac$, which is the well-known discriminant of the quadratic polynomial $p$.
  
  Each $0$-cycle $p = c$, regarded as the quadratic form $cY^2$ on $K^2$, has the same zero vectors $(x, 0)$, which represent the point at infinity in the projective line $\mathbf{P}(K^2)$. Each $2$-cycle with zero discriminant can be written as $a(X - u)(X - u) = a(X^2 - 2uX + u^2)$, and regarded as a quadratic form on $K^2$, that $2$-cycle has $(ux, x)$ as zero vectors.  Those zero vectors represent the finite point $(u, 1)$ of the projective line $\mathbf{P}(K^2)$.
\end{proof}

Based on the above description of isotropic cycles in $E$, the quadric of $E$ under the cycle pairing has a natural identification with the projective line $\mathbf{P}(K^2)$.

\section{Action of $SO(E)$ on the Quadric}
\label{sec:action}

Note that because $E$ has odd dimension, the map $v \mapsto -v$ has determinant $-1$.  Accordingly, the natural homomorphism of $SO(E)$ to the projective orthogonal group of $E$ is an isomorphism.  The natural action of $SO(E)$ on the projective quadric of $E$ is the same as the natural action of the projective orthogonal group, and such action is generated by reflections (or symmetries) defined by nonisotropic elements of $E$. 
As a threshold matter, we want to show that the action of the projective orthogonal group of $E$ on the quadric is faithful, meaning that the only projective orthogonal transformation of $E$ that leaves invariant all points of the quadric is the identify transformation.  Take the following cycles in $E$: 
\bigskip
\begin{itemize}
    \item $u = X.X$
    \item $v = (X - 1).(X - 1) = X.X - 2X + 1$
    \item $w = 1$
    \item $t = X.X - 4X + 4 = -u + 2v + 2w$
\end{itemize}
\bigskip

Note that $u$, $v$, and $w$ are a linear basis of the 3-dimensional vector space $E$, and $t$ can be expressed as a linear combination of $u$, $v$, and $w$ with no zero coefficient.  These 4 points therefore give us a projective frame of $E$.  Moreover, all these points are isotropic.  Therefore, if a projective transformation leaves all isotropic points in the projective space of $E$ invariant, it must be the identity transformation.

We now determine the action of a reflection, defined by a nonisotropic cycle in $E$, on the quadric.  

\begin{proposition}
  \label{prop:2}
  Given a nonisotropic cycle in $E$, the reflection defined by that cycle has the same action on the quadric of $E$ as the involution or inversive transformation defined by that cycle on the projective line $\mathbf{P}(K^2)$ = $K\cup \{\infty\}$, where the quadric of $E$ is identified with the projective line as discussed above.
\end{proposition}

\begin{proof}
  
  Let us first make the following key observation about the identification of the projective quadric of $E$ with the projective line $\mathbf{P}(K^2)$.  If $p$ is a cycle in $E$, then the zero points of $p$ in $\mathbf{P}(K^2)$, if any, are identified with the isotropic cycles of $E$ that are orthogonal to $p$ under the cycle pairing.
  
  Indeed, consider a $0$-cycle $p = c$.  Regarded as the quadratic form $cY^2$ on $K^2$, that cycle has only one zero point in $\mathbf{P}(K^2)$, namely the point at infinity represented by the vectors $(u, 0)$.  That point is identified with any $0$-cycle $r = u$.  It is clear that $\langle p, r \rangle$ = 0.
  
  Now consider any $1$-cycle $p$ = $bX + c$, with $b$ $\neq 0$.  That cycle has two zero points in $\mathbf{P}(K^2)$, namely the point at infinity represented by the vectors $(u, 0)$ and also the finite point represented by $(c, -b)$, which corresponds, up to a scalar factor, to the cycle $q = X^2 + (2c/b)X + (-c/b)^2$ under our identification.  The cycle pairing of $p$ with the cycle $r = u$ is clearly zero.  The cycle pairing of $p$ with the cycle $q = X^2 + (2c/b)X + (-c/b)^2$ is equal to $2c - 2c = 0$.  So our observation also holds in this case.
  
  Now consider a $2$-cycle $aX^2 + bX + c$, with $a$ $\neq 0$.  The point at infinity is not a zero point of that cycle.  If the polynomial $aX^2 + bX + c$ is not irreducible over $K$, so that the expression has roots $u$ and $v$ in $K$ (which may be the same if the root is repeated), then the cycle has zero points represented by $(u, 1)$ and $(v, 1)$.  Take the zero point $(u, 1)$  as example, it is represented by the isotropic cycle $q = X^2 - 2uX + u^2$.  The cycle pairing $\langle p, q \rangle = -2ub - 2au^2 - 2c = -2p(u) = 0$.  So the observation holds in all cases.
  
  Let $p$ be a nonisotropic cycle in $E$.  The reflection defined by $p$ is the linear transformation of $E$ given by the following well-known formula:         \begin{equation*}
         x \mapsto x' = x - \frac{2\langle x, p \rangle}{\langle p, p \rangle}p
   \end{equation*}
  
  This is a transformation of order 2, so the action of this reflection on the quadric is determined by the fixed points (if any), and pairs of points conjugate under the reflection.
  
  The reflection equation tells us that a point on the quadric  represented by $x$ is invariant under the reflection action if and only if $\langle x, p \rangle = 0$.  By our observation above, the isotropic point represented by $x$ must be a zero point of the cycle $p$ under our identification of the quadric with a projective line. That implies the reflection action on the quadric either has exactly two fixed points, or none at all.
  
  We will now show that each pair of conjugate points are the zero points of a cycle orthogonal to $p$ under the cycle pairing.  
  
  Let $u$ and $v$ be a pair of distinct conjugate points under the reflection by $p$.  Fist, we consider the case where both these points are finite.  
  
  The cycle $(X - u)(X - v)$ has $u$ and $v$ as zero points, and therefore it is orthogonal to the isotropic cycles $X^2 - 2uX + u^2$ and $X^2 - 2vX + v^2$.  Because reflection is an orthogonal transformation, it must map the cycle $(X - u)(X - v)$ to something that is orthogonal to both $X^2 - 2uX + u^2$ and $X^2 - 2vX + v^2$, i.e., to a cycle that has $u$ and $v$ as zero points.  But up to a scalar factor, $(X - u)(X - v)$ is the only cycle with that property.  Therefore,  $(X - u)(X - v)$ is an eigenvector of the reflection.  A reflection has eigenvectors with either eigenvalue 1 (for vectors in the orthogonal complement of $p$) or eigenvalue -1 (for vectors proportional to $p$).  That means $(X - u)(X - v)$ must be a vector in the orthogonal complement of $p$ because it clearly is not proportional to $p$ (not having the same zero points).
  
  Now consider the case where one of the conjugate points, say $v$, is the point at infinity.  In this case, $u$ and $\infty$ are the zero points of the $1$-cycle $X - u$.  Moreover, because $\infty$ is not a fixed point of the reflection defined by $p$, $p$ must be a $2$-cycle.  Let $p = aX^2 + bX + c$ with $a \neq 0$ and discriminant $<p, p> = b^2 - 4ac \neq 0$.  If we identify $\infty$ with a $0$-cycle $q = d$, then the reflection formula gives us
  \begin{equation*}
         q \mapsto q' = q - \frac{2\langle q, p \rangle}{\langle p, p \rangle}p = - saX^2 - sbX - sc + d, 
  \end{equation*}
where for abbreviation we have written  
$$ 
s = \frac{2\langle q, p \rangle}{\langle p, p \rangle} = -\frac{4ad}{b^2 - 4ac}. 
$$  
  
  Note that $q'$ is a 2-cycle, and of course isotropic because it is the transform of an isotropic cycle under an orthogonal mapping.  By looking at the coefficients of $X^2$ and $X$, we can readily see that this must be an isotropic 2-cycle with $-b/2a$ as zero point.  So we can take $u = -b/2a$.  The cycle pairing between the $1$-cycle $X + b/2a$ and $p = aX^2 + bX + c$ gives us $b - 2(b/2a)a = 0$.  
  
  Therefore the reflection defined by $p$ gives us an action on the quadric where the fixed points (if any) are the zero points of $p$ and the pairs of distinct conjugate points are zero points of cycles orthogonal to $p$.
  
  Proposition 1 of \cite{phat21} tells us that this action is exactly the same as the action defined by the involution or inversive transformation associated with $p$, under our identification of the quadric of $E$ with $\mathbf{P}(K^2)$.  
  \end{proof}
  
  Because involutions generate all the projective transformations of the projective line, the action of $SO(E)$ on the quadric is the same as the action of all projective transformations of the projective line.  That means the natural morphism of $SO(E)$ to $PGL(1, K)$ given by this geometric representation is surjective.  In addition, since the action of $SO(E)$ on the quadric of $E$ is faithful, as we showed earlier, that morphism of $SO(E)$ to $PGL(1, K)$ is also injective.  Therefore, the morphism is bijective and the above geometric representation gives us a natural isomorphism between $SO(E)$ and $PGL(1, K)$.

  The action of $SO(E)$ on the quadric is the same as the action of the projective orthogonal transformations, and each element of $SO(E)$ can be written as a product of two reflections according to the theorem of Cartan-Dieudonne (recall that dim$E$ = 3).  Therefore, our geometric representation provides a clear demonstration or explanation of the classical fact that each projective transformation in dimension 1 is a product of two involutions.  In fact, our geometric representation shows that each involution itself can also be written as a product of two other involutions, because the action of such an involution, once we identify it with a transformation in $SO(E)$, can be translated into a product of two reflections.

\section{A Natural Cross Ratio}
\label{sec:crossratio}

Following in the footsteps of Felix Klein, we are interested in finding numbers associated with configurations of points in the projective quadric of $E$ that are invariant under the action of the projective orthogonal group $PO(E)$.

If $u$ and $v$ are any representative vectors of two distinct points in the projective quadric, the pairing $\langle u, v \rangle$ is invariant under the action of $PO(E)$.  However, that pairing by itself cannot be an invariant, because different representative vectors will give us different values of the pairing due to different scalar factors.

Note that  $\langle u, v \rangle$ is always nonzero, because otherwise the two vectors $u$ and $v$ will generate a totally isotropic subspace of dimension 2, which by the Witt decomposition would mean $E$ contains the sum of two Artinian planes, which of course is impossible when dim$E$ = 3.  With that in mind, we can consider the following cross ratio associated to any four distinct points in the quadric, with representative vectors $x, y, z, t$:
$$
         \frac{\langle x, z \rangle \langle y, t \rangle}{\langle x, t \rangle \langle y, z \rangle}
$$

This cross ratio is well-defined because any pairing value of two distinct points is always nonzero as noted above, and moreover the cross ratio gives us the same value regardless of which representative vectors are chosen for the four distinct points due to cross cancellation of scalar factors.  Accordingly, this cross ratio is naturally invariant under the action of $PO(E)$.

Our geometric representation shows that projective transformations of the projective line are equivalent to orthogonal transformations of the quadric in a (regular) isotropic quadratic space of dimension 3.  For two finite points $u$ and $v$ on the projective line, if we take the $2$-cycles  $X^2 – 2uX + u^2$ and $X^2 – 2vX + v^2$  as isotropic cycles representing the corresponding points on the quadric of $E$, then their cycle pairing is $– 2(u – v)^2$.   That means the natural cross ratio for that quadric is equal to the square of the classical cross ratio of a projective line. 

Accordingly, our geometric representation gives us a new perspective on the classical cross ratio.  This new perspective also suggests the following proposition.

\begin{proposition}
  \label{prop:3}
  Any transformation of a projective line over a field K of characteristic $\neq 2$ that leaves invariant the square of the classical cross ratio must be a projective transformation.
\end{proposition}

\begin{proof}
Let $f$ be a transformation of the projective line $K\cup \{\infty\}$ that leaves invariant the square of the classical cross ratio.   We will show that if 0, 1 and $\infty$ are fixed by $f$, then $f$ must be the identity transformation.  This will prove the proposition, because any three distinct points can be mapped by a projective transformation to 0, 1 and $\infty$.
	
For any point $\lambda \neq 0, 1, 
\infty$, the classical cross ratio $[\lambda , 1; 0, \infty]$ = $\lambda$.  So $f(\lambda)$ = $\lambda$ or $f(\lambda)$ = $- \lambda$.  This implies that $f(–1) = –1$, as we already have $f(1) = 1$.

Let  $\lambda \neq 0, 1, -1, \infty$.  The cross ratio $[\lambda, – 1; 1, \infty] = (1 – \lambda)/2$.   If $f(\lambda) = – \lambda$, then the square of that cross ratio must be equal to the square of the cross ratio  $[-\lambda, -1; 1, \infty] = (1 + \lambda)/2$.  That means  $(1 – \lambda)^2 = ( 1 + \lambda)^2$  or  $–2\lambda = 2\lambda$, a contradiction as $\lambda \neq 0$.  Accordingly, $f(\lambda) = \lambda$ for all $\lambda$, i.e., $f$ is the identity transformation.
\end{proof}

We should mention that by the same reasoning, we have the same natural cross ratio for the quadric of any space $E$ of Witt index 1 (meaning the space is isomorphic to a sum of an Artinian plane and an anisotropic space, i.e., having no isotropic vectors).  The same invariant applies if we just work with a subset of that quadric stable under the action of a subgroup of $PO(E)$.  It happens that some well-known geometric spaces, such as the hyperbolic plane or hyperbolic spaces of higher dimensions, can be regarded as subset of such a quadric in a space of Witt index 1, and the relevant geometric transformation group can be identified with a subgroup of $PO(E)$.  In those situations, the cross ratio described here gives a natural invariant for the geometry.  In addition, if we restrict to certain subgroups of $PO(E)$, other invariants may also come up in a natural way.  For example, if we look at the geometry defined by the subgroup of orthogonal transformations leaving invariant a given nonisotropic vector $p$ of $E$, then the expression 
$$\frac{\langle x, p \rangle \langle y, p \rangle}{\langle x, y \rangle}$$
clearly gives us another natural invariant for that geometry.  I hope that the interested reader will have a happy time exploring these questions.

\end{document}